\documentclass[12pt,leqno]{amsart}  
\usepackage{amsmath,amstext,amsthm,amssymb,amsxtra}
\usepackage{txfonts} % pxfonts txfonts 
\usepackage[T1]{fontenc}
\usepackage{lmodern}

\usepackage{euler}   % better than the option below
\usepackage{latexsym}

\usepackage{tikz}

\newcommand{\art}[6]{{\sc #1, \rm #2, \it #3 \bf #4 \rm (#5), \mbox{#6}.}}
\newcommand{\book}[3]{{\sc #1, \it #2, \rm #3.}}
\newcommand{\AND}{{\rm and }}

\newcommand{\loc}{_{\rm loc}}
\DeclareMathOperator*{\diam}{diam}
\DeclareMathOperator*{\dist}{dist}

\newcommand{\R}{\mathbb{R}}

\newcommand{\N}{\mathbb{N}}
\newcommand{\Z}{\mathbb{Z}}

\def\diam{\qopname\relax o{diam}}

\def\dist{\qopname\relax o{dist}}

\def\b{\qopname\relax o{b}}

\numberwithin{equation}{section}

\def\kint_#1{\mathchoice%
          {\mathop{\kern 0.2em\vrule width 0.6em height 0.69678ex depth -0.58065ex
                  \kern -0.8em \intop}\nolimits_{\kern -0.4em#1}}%
          {\mathop{\kern 0.1em\vrule width 0.5em height 0.69678ex depth -0.60387ex
                  \kern -0.6em \intop}\nolimits_{#1}}%
          {\mathop{\kern 0.1em\vrule width 0.5em height 0.69678ex depth -0.60387ex
                  \kern -0.6em \intop}\nolimits_{#1}}%
          {\mathop{\kern 0.1em\vrule width 0.5em height 0.69678ex depth -0.60387ex
                  \kern -0.6em \intop}\nolimits_{#1}}}

\usepackage{hyperref} %,hypertexnames=false,colorlinks,[pagebackref]
\hypersetup{
    colorlinks=true,       % false: boxed links; true: colored links
    linkcolor=blue,          % color of internal links
    citecolor=magenta,        % color of links to bibliography
    filecolor=magenta,      % color of file links
    urlcolor=cyan           % color of external links
}

\allowdisplaybreaks
\swapnumbers

\theoremstyle{plain} % definition 
\newtheorem{lemma}[equation]{Lemma} 
\newtheorem{proposition}[equation]{Proposition} 
\newtheorem{theorem}[equation]{Theorem} 
\newtheorem{corollary}[equation]{Corollary}

\theoremstyle{definition}
\newtheorem{definition}[equation]{Definition} 

\theoremstyle{remark}
\newtheorem{remark}[equation]{Remark}

\newtheorem*{ack}{Acknowledgment}

\numberwithin{equation}{section}

\title[Aspects of local to global results]{Aspects of local to global
  results}
\author{Ritva Hurri-Syrj\"anen}
\address[R.H.-S.]{University of Helsinki, Department of Mathematics
  and Statistics, P.O. Box 68 (\def\b{\qopname\relax o{b}}Gustaf H\"allstr\"ominkatu 2 $\b$),
  FI-00014 University of Helsinki, Finland}
\email{ritva.hurri-syrjanen@helsinki.fi} 

\author{Niko Marola}
\address[N.M.]{University of Helsinki, Department of Mathematics and
  Statistics, P.O. Box 68 (\def\b{\qopname\relax o{b}}Gustaf H\"allstr\"ominkatu 2 $\b$), FI-00014
  University of Helsinki, Finland} 
\email{niko.marola@helsinki.fi}

\author{Antti V. V\"ah\"akangas} \address[A.V.V.]{University of
  Helsinki, Department of Mathematics and Statistics, P.O. Box 68
  (\def\b{\qopname\relax o{b}}Gustaf H\"allstr\"ominkatu 2 $\b$), FI-00014 University of Helsinki,
  Finland} \email{antti.vahakangas@helsinki.fi} \thanks{A.V.V. benefited from
  the research program {\it Operator Related Function Theory and
  Time-Frequency Analysis} at the Centre for Advanced Study of the
  Norwegian Academy of Science and Letters in Oslo during 2012--2013.
  A.V.V. was  supported by the Finnish Academy of Science and
  Letters, Vilho, Yrj\"o and Kalle V\"ais\"al\"a Foundation}

\begin{document}

\begin{abstract} 
  We establish %Reimann and Rychener type
  local to global results for a function space which is larger than
  the well known BMO space, and was also introduced by John and
  Nirenberg.
%  We also prove local to global results for weak type inequalities.
  \end{abstract}

\maketitle
	
        % \setcounter{tocdepth}{1}
        % \tableofcontents
     

%\subjclass{Primary: 42B35; Secondary: 46E30.}

%%%%%%%%%%%%%%%%%%%%%%%%%%%%%% SECTION  SECTION SECTION
%%%%%%%%%%%%%%%%%%%%%%%%%%%%%% SECTION  SECTION SECTION 
\section{Introduction}
\label{sect:JN}

The space of functions of bounded
mean oscillation, abbreviated to
BMO, is introduced by John and Nirenberg \cite{JN}.
In the same paper, John and Nirenberg introduced a larger space of functions.
As opposed to any BMO function, that has exponentially
decaying distribution function, a function in this larger space  is
known to  belong to a weak $L^p$-space, \cite[Lemma 3]{JN}; the
inclusion being strict, see \cite[Example 3.5]{ABKY}.
We extend this weak-type inequality to the case of John domains.
The equivalence of local and global BMO norms is a rather well-known result, due to Reimann and Rychener~\cite{RR}.
We obtain the corresponding local to global result 
for the mentioned larger space of functions.

Let $G$ be a proper open subset of $\R^n$, $n\geq 1$.  The following
condition was introduced in \cite{JN}: Let $f\colon G \to \R$ be a
function in $L^1(G)$ and let us assume that there exists
$1<p<\infty$ such that
\begin{equation} \label{eq:JNnorm} \mathcal{K}_f^p(G) :=
  \sup_{\mathcal{P}(G)}\sum_{Q\in
    \mathcal{P}(G)}|Q|\left(\fint_Q|f(x)-f_Q|\, dx\right)^p <\infty\,,
\end{equation}
where the supremum is taken over all partitions $\mathcal{P}(G)$ of
$G$ into cubes such that $Q\subset G$ for each $Q\in\mathcal{P}(G)$,
the interiors of these cubes are pairwise disjoint, and
$G= \bigcup_{Q\in\mathcal{P}(G)}Q$. We call such partitions
admissible. 

It is shown in \cite[Lemma 3]{JN} that a function satisfying
\eqref{eq:JNnorm}, with $G$ being a cube $Q$ in $\R^n$, belongs to a
weak $L^p(Q)$-space. More precisely, there exists a positive constant
$C$, depending only on $n$ and $p$, so that for all $f\in L^1(Q)$,
\begin{equation} \label{eq:weaktype} \sigma^p\left|\{x\in Q:\
    |f(x)-f_{Q}|>\sigma\}\right| \leq C\mathcal{K}^p_f(Q)
\end{equation}
for each $\sigma>0$. We refer to \cite{Giusti, Tor, ABKY} for
other proofs of this result. 

We mention papers \cite{FPW, FPW03, MacP, MacP2} where a related
discrete summability condition is studied, and a recent paper
\cite{BM} where its relation to condition \eqref{eq:JNnorm} is
discussed. In \cite{FPW}, in particular, the authors prove a local to
global result in connection with this discrete summability
condition. However, the approach considered in the present paper is
different from the one in \cite{FPW} and of independent interest.

Let us localize condition \eqref{eq:JNnorm} in the following
way. For a function $f\in L\loc^1(G)$, we define the number
\begin{equation} \label{eq:JNloc} \mathcal{K}_{f,\textrm{loc}}^p(G) :=
  \sup_{\mathcal{P}\loc(G)}\sum_{Q\in\mathcal{P}\loc(G)}|Q|\left(\fint_Q|f(x)-f_Q|\,
    dx\right)^p,
\end{equation}
where the supremum is taken over all partitions
$\mathcal{P}\loc(G)$ of $G$ into cubes such that for each $Q\in\mathcal{P}\loc(G)$ %the interiors of cubes are pairwise disjoint,
 a dilated cube $\lambda Q\subset G$, with fixed $\lambda>1$, and
 these cubes have bounded overlap, specifically,
\[\sup_{x\in G} \sum_{Q\in\mathcal{P}\loc(G)}\chi_Q(x) \leq N\,,
\]
where $N\ge 1$ is a finite constant depending on $n$ only.
We call such
partitions local.

We shall prove a Reimann--Rychener-type local to global result. % for
%functions satisfying condition \eqref{eq:JNloc}. 
More precisely, in
Theorem~\ref{t.local-to-global},
% the equivalence of the preceding local
%and global quantities are shown; that is,
we show that there exists a positive
constant $C$, depending on $n$, $p$, and $\lambda$, such that  for all $f\in L^1(G)$
\[
\mathcal{K}_f^p(G) \leq  C\mathcal{K}_{f,\textup{loc}}^p(G)\,.%\le C\mathcal{K}_f^p(G)\,.
\]
In the second part of the paper, we consider necessary and sufficient
conditions for Euclidean domains to support the weak-type inequality \eqref{eq:weaktype}. 
Our
main results are stated in Theorem~\ref{t.weak_john} and
Theorem~\ref{t.necessity}.

\begin{ack}
  The authors would like to thank Juha Kinnunen for  valuable
  discussions on the subject and for pointing out the reference
  \cite{BM}. The authors would also like to thank Juha Lehrb\"ack for
  pointing out \cite[Lemma 6]{HK} to us. 
\end{ack}

\section{Notation and preliminaries}
\label{s.notation}

Throughout the paper, a cube $Q$ in $\R^n$ is a closed cube with sides
parallel to the coordinate axes. For a cube $Q$, with side length
$\ell(Q)$, and for $\lambda>0$, we write the dilated cube, with side
length $\lambda \ell(Q)$, as $\lambda Q$. We write $\chi_A$ for the
characteristic function of a set $A$, the boundary of $A$ is written
as $\partial A$, and $\lvert A\rvert$ is the Lebesgue $n$-measure of a
measurable set $A$ in $\R^n$.
%The interior of a cube $Q$ is written as $\inte(Q)$. 
The integral average of $f\in L\loc^1(\R^n)$ over a bounded set $A$ with positive measure is written as
$f_A$, that is,
\[
f_A = \fint_Af\, dx = \frac1{\rvert A \lvert}\int_A f\, dx\,.
\] 
Various constants whose value may change even within
a given line are denoted by $C$. 

The family of closed dyadic cubes is written as $\mathcal{D}$. We let
$\mathcal{D}_j$ be the family of those dyadic cubes whose side length
is $2^{-j}$, $j\in\Z$. For a proper open set $G$ we fix its
Whitney decomposition $\mathcal{W}(G)\subset\mathcal{D}$, and write
$\mathcal{W}_j(G)=\mathcal{D}_j\cap \mathcal{W}(G)$. 
For a Whitney
cube $Q\in\mathcal{W}(G)$ we write $Q^*=\frac{9}{8} Q$. Such dilated cubes 
have a bounded overlap, with upper bound depending on $n$ only, and
they satisfy
\begin{equation}\label{dist_est}
  \frac{3}{4}\diam(Q)\le \dist(x,\partial G)\le 6\diam(Q),
\end{equation}
whenever $x\in Q^*$. 
%In particular,
%\begin{equation}\label{e.finite}
%\sup_{Q\in\mathcal{W}(G)} \sharp\{ R\in\mathcal{W}(G)\,:\,R^*\cap Q^*\not=\emptyset\}\le N\,,
%\end{equation}
%with $N$ depending on $n$ only.
%We will assume that $\diam(Q)\le 1$ for every $Q\in\mathcal{W}(G)$; this is done by
%always restricting ourselves to sets $G$ whose diameter is bounded by one.
%\[
%\sup_{x\in G}\sum_{Q\in\mathcal{W}(Q_0)}\chi_{Q^*}(x) \leq 12^n<\infty;
%\]
For other properties of Whitney cubes we refer to \cite[VI.1]{Stein}.

For a bounded domain $G$ in $\R^n$, we will construct a {\em chain} of cubes
\[
\mathcal{C}(Q)= (Q_0,\ldots, Q_k)\subset \mathcal{W}(G)\,,
\]
joining $Q_0$ and $Q=Q_k$, such that 
 $Q_i\not=Q_j$ whenever $i\not=j$,
 and there exists a positive finite constant
$C=C(n)$ for which 
\begin{equation}\label{p.2}
\lvert Q_j^* \cap Q_{j-1}^*\rvert \ge C\max\{\lvert Q_j^*\rvert, \lvert Q_{j-1}^*\rvert\}
\end{equation}
with each $j\in \{1,\ldots,k\}$. A given family
$\{\mathcal{C}(Q):\ Q\in\mathcal{W}(G)\}$ with a fixed Whitney cube $Q_0$
is a {\em chain decomposition} of $G$.  
A  {\em shadow} of a Whitney cube $R\in\mathcal{W}(G)$ is the set
\[
\mathcal{S}(R) = \{Q\in\mathcal{W}(G):\ R\in\mathcal{C}(Q)\}\,.
\]
% The preceding construction of chains of dilated cubes can be
% organized so that
% \begin{equation} \label{eq:Shadow} \sum_{Q'\in \mathcal{S}(Q)}|Q'|
%   \leq C|Q|,
%\end{equation}
%for every $Q\in\mathcal{W}(Q_0)$, where the positive constant $C$
% depends on $n$ only. We refer, e.g., to \cite{H}.
%
% Throughout the paper, ${G}$ and $H$ are bounded domains in
%$\R^n$, $n\geq 2$. 
%For a measurable set $E$, with positive and finite Lebesgue measure $\lvert E\rvert>0$,
%we write as usual
%\[
%f_E = \fint_E f(x)\,dx = \frac{1}{\lvert E\rvert}\int_E f(x)
% \,dx\,,\qquad f\in L^1_{\textup{loc}}(\R^n)\,.
%\]
%The characteristic function of $E$ is written as $\mathbf{1}_E$.
%\smallskip
% Without essential loss of generality, we may assume that
% $\mathrm{diam}(Q)\le 1$ for every $Q\in\mathcal{W}$.
%

Let us recall the definition of John domains.  The condition in
Definition \ref{sjohn} was first used by John in \cite{J}; the
connection of this condition and the theory of Poincar\'e and Sobolev
type estimates was apparently first introduced by Boman in his
unpublished paper \cite{Boman}.

\begin{definition}\label{sjohn}
  A bounded domain $G$ in $\R^n$, $n\ge 2$, is a {\em John domain}, if
  there exist a point $x_0\in G$ and a constant $\beta_G\ge 1$ such that every
  point $x$ in $G$ can be joined to $x_0$ by a rectifiable curve
  $\gamma:[0,\ell]\to G$ parametrized by its arc length for which
  $\gamma(0)=x$, $\gamma(\ell)=x_0$, $\ell \le \beta_G \diam(G)$, and for all $t\in [0,\ell]$,
\[
\dist(\gamma(t),\partial G)\ge t/\beta_G\,.
\]
The point $x_0$ is called a {\em John center}
of $G$, and the smallest constant $\beta_G\ge 1$ is called the {\em
  John constant} of $G$.
\end{definition}

% $\beta_G=\beta_{\lambda G}$ for all $\lambda>0$.
Bounded Lipschitz domains and bounded domains with the interior cone condition
are John domains. Also, the Koch snowflake is a John domain in the plane.
Observe that the John constant is invariant under scaling 
and translation of $G$.

The following observation concerning a given John domain $G$ will be
relevant to us.  There exist a positive number $s=s(n,\beta_G)<n$ and
a constant $C=C(n,\beta_G)>0$, such that
\begin{equation}\label{e.assouad}
\int_{B(y,r)} \mathrm{dist}(x,\partial G)^{s-n}\,dx\le C r^s
\end{equation}
for every $y\in \partial G$ and for every $r>0$.  Inequality \eqref{e.assouad}
is essentially covered by
\cite[Lemma 6]{HK}, but
%Aside from the dependencies of $s$ and $C$, 
it is also an immediate consequence of the
following three facts:
\begin{itemize}

\item[(1)] the boundary $\partial G$ of a John domain is porous in $\R^n$; 

\item[(2)] the Assouad dimension of a porous set in $\R^n$ is
  strictly less than $n$, \cite{Luukkainen}; 

\item[(3)] the Assouad dimension of $\partial G$ coincides with the
  Aikawa dimension of $\partial G$; we refer to a recent paper \cite{lehrbackII}.  

\end{itemize}
Indeed, by (1)--(3), the Aikawa dimension of $\partial G$ is strictly
less than $n$, and inequality \eqref{e.assouad} follows. The fact that both
$s$ and $C$ can be chosen, depending on $n$ and $\beta_G$ only, is
straightforward but tedious to verify. We omit the details. 

The following proposition provides a chain decomposition of a given
John domain. From now on, any reference to a chain decomposition will
be to the one presented in Proposition~\ref{p.admissible}.

\begin{proposition}(Chain decomposition)\label{p.admissible} Suppose
$1<p<\infty$ and ${G}$ is a John domain in $\R^n$.  Then
  there exist constants $\sigma,\tau\in\N$ and a chain
  decomposition $\{\mathcal{C}(Q):\ Q\in\mathcal{W}(G)\}$ of $G$ with
  the following conditions (1)--(3):
\begin{itemize}
\item[(1)] $\ell(Q)\le 2^{\tau}\ell(R)$ for each $R\in\mathcal{C}(Q)$
  and $Q\in\mathcal{W}(G)$;
\item[(2)] $\sharp \{R\in\mathcal{W}_j(G):\ R\in\mathcal{C}(Q)\}\le
 2^ \tau$ for each $Q\in \mathcal{W}(G)$ and $j\in \Z$;
\item[(3)] The following inequality holds,
\begin{equation}\label{e.useful}
\sup_{j\in \Z}\sup_{R\in\mathcal{W}_j(G)}
\frac{1}{\lvert R\lvert}
\sum_{k=j-\tau}^\infty \sum_{\substack{Q\in \mathcal{W}_k(G) \\ Q\in\mathcal{S}(R) }} \lvert Q\rvert
(\tau+1+k-j)^{p}  < \sigma\,.
\end{equation}
\end{itemize}
Furthermore, the constants $\sigma$ and $\tau$ depend only
on $n$, $p$, and the John constant $\beta_G$.
\end{proposition}

\begin{proof}
Let us first construct a chain decomposition
of $G$. 
We fix a Whitney cube $Q_0$ containing the 
John
center $x_0$ of $G$.
Let $Q\in\mathcal{W}(G)$ and let us fix a rectifiable curve $\gamma$
that is parametrized by its arc length and joins the midpoint $x_Q$ of
$Q$ and $x_{0}$ as in Definition \ref{sjohn}. 

 First assume that $Q\cap Q_0\not=\emptyset$.
Then, we join  $x_Q$ to the  midpoint $x_{Q_0}$ of $Q_0$ by an arc that is contained in
$Q\cup Q_0$ and whose length is comparable to $\ell(Q)$.  Otherwise
there is $r>0$ such that $\gamma(r)$ lies in the boundary of a Whitney
cube $P$ that intersects $Q$ and $\gamma(t)$ belongs to a cube that is
not intersecting $Q$ whenever $t\in (r,\ell(\gamma)]$.  Join
 $x_Q$ to  $x_P$ by an arc whose length is
comparable to $\ell(Q)$ and is in $Q\cup P$.  We iterate these steps
with $Q$ replaced by $P$, and we continue until we reach $x_{Q_0}$.
Let $\gamma_Q$ be this composed curve parametrized by its arc length.

It is straightforward to verify that there
is a constant $\rho\ge 1$, depending on $n$ and $\beta_G$, such that
%$\ell(\gamma_Q)\le \rho$ 
for every $t\in[0,\ell(\gamma_Q)]$,
\begin{equation}\label{esg}
  \dist(\gamma_Q(t),\partial G) \geq t/\rho\,.
\end{equation}
Let $\mathcal{C}(Q)$
be the chain consisting of cubes $R\in \mathcal{W}(G)$ such that the
midpoint $x_R= \gamma_Q(t_R)$ for some $t_R\in[0,\ell(\gamma_Q)]$.  
%It is clear that these disjoint cubes can be organized in such a way that
%\eqref{p.2} is satisfied and $Q_0$ is the first cube in the chain.

We  verify that this chain decomposition of $G$ satisfies
conditions (1)--(3).

{\bf{Condition (1):}} Let $Q\in\mathcal{W}(G)$ and $R\in\mathcal{C}(Q)$. Clearly,
we may assume that $R\not=Q$. Hence, if $\gamma_Q(t_R)=x_R$, then by
inequalities \eqref{esg} and \eqref{dist_est}, 
\[
\ell(Q)/2 \le t_R \le \rho\dist(\gamma_Q(t_R), \partial G) =
\rho\dist(x_R, \partial G)\le 6 \rho\sqrt n \,\ell(R)\,.
\]

{\bf{Condition (2)}:} Let $Q\in\mathcal{W}(G)$ and $j\in \Z$. 
%We shall prove
%the inequality 
%\begin{equation}\label{eks} \sharp
%  \{R\in\mathcal{W}_{j}(G):\ R\in \mathcal{C}(Q)\}\le 6\rho \sqrt n +
%  1.
%\end{equation}
Let $R_1,\ldots,R_M\in\mathcal{W}_{j}(G)$ be cubes such that $R_i\in
\mathcal{C}(Q)$ for every $i\in \{1,\ldots,M\}$.  We number these
cubes in the same order as $\gamma_Q$ hits their midpoints. In
particular, if $\gamma_Q(t)=x_{R_M}$, then $\gamma_Q([0,t])$ joins the
midpoints of $M$ cubes whose side length is $2^{-j}$. By \eqref{esg}
and \eqref{dist_est},
\[
(M-1)2^{-j}\le t \le \rho \dist(\gamma_Q(t),\partial G) =  \rho\dist(x_{R_M},\partial G)
\le 6\rho \sqrt n\,2^{-j}.
\]
It follows that $M\le 6\rho \sqrt n+1$, hence we obtain condition (2).
%as required in \eqref{eks}.

Let us fix $\tau=\tau(n,\beta_G)\in\N$ for which both 
conditions (1) and (2) are valid.

{\bf{Condition (3):}} Let us first prove that there is a constant
$C=C(n,\beta_G)>0$ such that, for each $R\in\mathcal{W}(G)$,
\begin{equation}\label{e.standard}
\bigcup_{Q\in \mathcal{S}(R)} Q \subset B(y_R, C\ell(R)),
\end{equation}
where $y_R\in\partial G$ is any point satisfying $\lvert x_R-y_R\rvert
=\dist(x_R,\partial G)$. Consider any cube $Q\in\mathcal{S}(R)$.
Since $R\in\mathcal{C}(Q)$, there is $t_R\in [0,\ell (\gamma_Q)]$ such
that $x_R=\gamma_Q(t_R)$. Hence, if $x\in Q$,
\begin{align*}
\lvert x-y_R\rvert  \le \lvert x-x_Q\rvert + \lvert x_Q-x_R\rvert + \lvert x_R-y_R\rvert\,.
\end{align*}
Observe that $\lvert x-x_Q\rvert \le \diam(Q)\le 2^\tau \diam(R)$ and
$\lvert x_R-y_R\rvert \le 6\diam(R)$. By  inequality~\eqref{esg},
\begin{align*}
\lvert x_Q-x_R\rvert  &= \lvert \gamma_Q(0)-\gamma_Q(t_R)\rvert
\le t_R \le \rho\dist(\gamma_Q(t_R),\partial G) \le 6\rho\diam(R)\,.
\end{align*}
Relation \eqref{e.standard} follows from the previous estimates.

Let $\epsilon=n-s>0$, where $s=s(n,\beta_G)$ is given by
\eqref{e.assouad}; recall that
$s$ is  related to  the Aikawa dimension of $\partial G$.  Fix $j\in \Z$ and $R\in\mathcal{W}_j(G)$. Then, if
$k\ge j-\tau$ and $Q\in\mathcal{W}_k(G)$,
\begin{equation} \label{i.test}
  \left(\frac{\ell (Q)}{\ell (R)}\right)^{\epsilon} (\tau + 1 + k - j)^p = 2^{(\tau +1)\epsilon}
  2^{-(\tau + 1+k-j)\epsilon} (\tau + 1 + k-j)^p \le C2^{\tau \epsilon},
\end{equation}
where $C=C(\epsilon,p)>0$. By inequality~\eqref{i.test},
\begin{align*}
  \sum_{k=j-\tau}^\infty \sum_{\substack{Q\in \mathcal{W}_k(G)
      \\ Q\in\mathcal{S}(R) }}
  \left(\frac{\ell(Q)}{\ell(R)}\right)^{n} (\tau+1+k-j)^{p} \le
  C2^{\tau \epsilon}
  \ell(R)^{-(n-\epsilon)}\sum_{Q\in\mathcal{S}(R)}
  \ell(Q)^{n-\epsilon}.
\end{align*}
On the other hand, by \eqref{dist_est}, \eqref{e.standard}, and \eqref{e.assouad}, we may
conclude that
\begin{equation*}
  \sum_{Q\in\mathcal{S}(R)} \ell(Q)^{n-\epsilon} 
%
% &\le c_{n,s,p,\epsilon} \int_{\cup \mathcal{S}(R)} \mathrm{dist}(x,\partial G)^{(n-sp-\epsilon)-n}\\
%
\le C\int_{B(y_R,C\ell(R))} \dist(x,\partial
G)^{s-n}\,dx \leq C\ell(R)^{n-\epsilon},
\end{equation*}
where $C=C(n,\epsilon,\beta_G)>0$, and condition (3) follows.
%The last term is bounded by $c_{n,\epsilon,\beta_G}
%\ell(R)^{n-\epsilon}$, inequality \eqref{e.assouad}.
\end{proof}

\section{A local to global result}
\label{sect:LocGlo}

In this section, we prove the following Reimann--Rychener-type local to global result.
%Let us define
%$\mathcal{K}^p_{f,John}(G)$ as in \eqref{eq:JNnorm} but now the
%supremum is taken over all partitions $\mathcal{P}_{\textup{John}}(G)$
%of $G$ into John domains $H$, such that $H\subset G$ for each
%$H\in\mathcal{P}_{\textup{John}}(G)$, whose John constant $\beta_H$ is
%bounded by $N$, and the John domains have bounded overlap,
%\[
%\sum_{H\in\mathcal{P}_{\textup{John}}(G) } \chi_H \le
%N,
%\]
%where $N=N(\mathcal{P}_{\textup{John}}(G))\ge 1$ is a constant. We call
%such partitions John partitions.

\begin{theorem}\label{t.local-to-global}
  Suppose $G$ is a proper open subset of $\R^n$, $n\ge 2$. If $f\in
  L^1(G)$ and $1<p<\infty$, then
\begin{equation}\label{e.eqs}
\mathcal{K}_{f}^p(G)%\le \mathcal{K}^p_{f,John}(G)
\leq  C\mathcal{K}^p_{f,\textup{loc}}(G)\,,
\end{equation}
where a positive constant $C$ depends on $n$, $p$, and $\lambda$.
\end{theorem}

%The first estimate in \eqref{e.eqs} follows directly from the
%definitions, with  $N$ depending on $n$ only, so it remains to prove the second inequality. 
%In the sequel, we will assume that $\diam(G)\le 1$; This is be established
%by observing that both sides of \eqref{e.eqs} are $n$-homogeneous
%with respect to scaling of $G$ by positive constants.

Let us
begin with a preliminary lemma, which is useful also in Section \ref{s.suff}.

\begin{lemma}\label{l.local-to-global}
  Let $H$ be a John domain in $\R^n$, $f\in
  L^1(H)$, and $1<p<\infty$. Then
\begin{align*}
  \left(\fint_H \lvert f(x)-f_{Q_0^*}\rvert \,dx\right)^p & + \left(\fint_H \lvert f(x)-f_H\rvert \,dx\right)^p \\
  &\leq \frac{C}{\lvert H\rvert}\sum_{Q\in\mathcal{W}(H)} \lvert Q^*\rvert \bigg(\fint_{Q^*}
  \lvert f(x)-f_{Q^*}\rvert\,dx\bigg)^p\,,
\end{align*}
where $Q_0$ is the fixed cube in the chain decomposition of $H$.
Moreover, a positive constant $C$ depends on $n$, $p$, and 
the John constant $\beta_H$. 
\end{lemma}

\begin{proof}
  Observe that
\begin{align}\label{e.est}
\int_{H} \lvert f(x)-f_H\rvert\,dx
&\le 2 \int_H \lvert f(x)- f_{Q_0^*}\rvert\,dx \notag \\
&\le 2 \sum_{Q\in\mathcal{W}(H)} \int_{Q^*} \lvert f(x)-f_{Q^*}\rvert\,dx
+ 2\sum_{Q\in\mathcal{W}(H)}\lvert Q\rvert \lvert f_{Q^*} - f_{Q_0^*}\rvert\,.
\end{align}
Let us estimate the first term on the right-hand side in \eqref{e.est}.
By H\"older's inequality,
\begin{equation}\label{e.first}
\begin{split}
  \sum_{Q\in\mathcal{W}(H)} &\int_{Q^*} \lvert f(x)-f_{Q^*}\rvert\,dx \\
  & \leq C\left(\sum_{Q\in\mathcal{W}(H)} \lvert Q\rvert\right)^{1/p'}
  \left(\sum_{Q\in\mathcal{W}(H)} \lvert Q\rvert \left(\fint_{Q^*}
      \lvert
      f(x)-f_{Q^*}\rvert\,dx\right)^p \right)^{1/p}\\
  & \le C\lvert H\rvert^{1/p'}\left(\sum_{Q\in\mathcal{W}(H)} \lvert
    Q^*\rvert \left(\fint_{Q^*} \lvert f(x)-f_{Q^*}\rvert\,dx\right)^p
  \right)^{1/p},
\end{split}
\end{equation}
where $p'=p/(p-1)$ is the conjugate exponent to $p$.

To estimate the second term on the right-hand side in \eqref{e.est},
we use a chain $\mathcal{C}(Q)=(Q_0,\ldots,Q_k)$ joining the cube
$Q_0$ to $Q_k=Q\in\mathcal{W}(H)$. Hence,
\begin{equation} \label{e.sss}
\sum_{Q\in\mathcal{W}(H)}\lvert Q\rvert \lvert f_{Q^*} - f_{Q_0^*}\rvert
\le \sum_{Q\in\mathcal{W}(H)}\lvert Q\rvert
\sum_{i=1}^k  \lvert f_{Q^*_i} - f_{Q_{i-1}^*}\rvert.
\end{equation}
Here, by property \eqref{p.2}, for any $i\in \{1,\ldots,k\}$
\begin{align*}
  \lvert f_{Q^*_i} - f_{Q_{i-1}^*}\rvert & \leq  \fint_{Q_i^*\cap Q_{i-1}^*}|f-f_{Q_i^*}|\, dx +  \fint_{Q_i^*\cap Q_{i-1}^*}|f-f_{Q_{i-1}^*}|\, dx \\
  & \leq C\sum_{j=i-1}^i \fint_{Q_j^*} \lvert
  f(x)-f_{Q_j^*}\rvert\,dx\,.
\end{align*}
By the fact that there are no duplicates in $\mathcal{C}(Q)$, i.e., 
$Q_i\not=Q_j$ if $i\not=j$, we obtain
\begin{align*}
  \sum_{Q\in\mathcal{W}(H)}\lvert Q\rvert \lvert f_{Q^*} -
  f_{Q_0^*}\rvert & \le C\sum_{Q\in\mathcal{W}(H)}\lvert Q\rvert
  \sum_{i=1}^k  \sum_{j=i-1}^i \fint_{Q_j^*} \lvert f(x)-f_{Q_j^*}\rvert\,dx \\
  & \le C\sum_{Q\in\mathcal{W}(H)}\lvert Q\rvert
  \sum_{R\in\mathcal{C}(Q)} \fint_{R^*} \lvert f(x)-f_{R^*}\rvert\,dx \\
  &\le C\sum_{R\in\mathcal{W}(H)} \sum_{Q\in \mathcal{S}(R)} \lvert
  Q\rvert
  \fint_{R^*} \lvert f(x)-f_{R^*}\rvert\,dx \\
  & \leq C\sum_{R\in\mathcal{W}(H)}\int_{R^*}|f(x)-f_{R^*}|\, dx\,,
\end{align*}
where  the last
inequality is a consequence of inequality \eqref{e.standard}. We may estimate as in connection with \eqref{e.first}. This completes the proof.
\end{proof}

\begin{remark}\label{r.cubes}
  The following inequality, interesting as such, follows from 
Lemma~\ref{l.local-to-global}. Let $Q$ be a cube 
 and
  $f\in L^1(Q)$. Then, for every $1<p<\infty$,
\begin{equation*}
 \left(\fint_{Q}|f(x)-f_{Q}|\, dx\right)^p \leq 
 \frac{C}{\lvert Q\rvert}\sum_{R\in\mathcal{W}(Q)}|R^*|\left(\fint_{R^*}|f(x)-f_{R^*}|\, dx\right)^p,
\end{equation*}
where $\mathcal{W}(Q)$ refers to Whitney decomposition of the
interior of $Q$ and $C$ is a positive constant depending only on
$n$ and $p$.
% Indeed, the john constant
%$\beta_{Q}$ depends only on $n$. 
\end{remark}

\begin{proof}[Proof of Theorem~\ref{t.local-to-global}]
%The second inequality in \eqref{e.eqs} follows easily from definitions;
%of particular importance is the oration that
%by inequality \eqref{e.finite_degree}, any $\mathcal{P}_{f,\textup{loc}}(G)$ can be partitioned
%into $N+1$ disjoint families.
  Let us  fix an admissible partition $\mathcal{P}(G)$
  of $G$ into cubes.  For each cube $Q\in
  \mathcal{P}(G)$ we form a local partition
 $\mathcal{P}_{\textup{loc}}(Q)=\{R^*:\
  R\in\mathcal{W}(Q)\}$. We write
\[
\mathcal{P}_{\textup{loc}}(G) = \bigcup_{Q\in\mathcal{P}(G)}
\mathcal{P}_{\textup{loc}}(Q).
\]
It is straightforward to verify that $\mathcal{P}_{\textup{loc}}(G)$
is a local partition of $G$. In particular, 
for each $R^*\in\mathcal{P}_{\textup{loc}}(Q)$ 
with $Q\in\mathcal{P}(G)$,
the inclusions $\lambda
R^*\subset Q\subset G$ are valid for $1<\lambda<\frac{10}{9}$. By
applying Remark~\ref{r.cubes} and observing that for each
$R^*\in\mathcal{P}_{\textup{loc}}(G)$  there is at most one cube
$Q\in\mathcal{P}(G)$ such that $R^*\in
\mathcal{P}_{\textup{loc}}(Q)$, we obtain
\begin{align*}
\sum_{Q\in \mathcal{P}(G) }
&\lvert Q\rvert \left(\fint_Q \lvert f(x)-f_Q\rvert\,dx\right)^p \\
& \le C\sum_{ Q\in \mathcal{P}(G) }
\sum_{R^*\in\mathcal{P}_{\textup{loc}}(Q)}
\lvert {R^*}\rvert \left(\fint_{R^*} \lvert f(x)-f_{R^*}\rvert\,dx\right)^p \\
& \le C\sum_{R^*\in\mathcal{P}_{\textup{loc}}(G)}
\lvert R^*\rvert \left(\fint_{R^*} \lvert f(x)-f_{R^*}\rvert\,dx\right)^p\le C
\mathcal{K}^p_{f,\textup{loc}}(G)\,.
\end{align*}  
The proof is completed by taking the supremum over all 
admissible partitions $\mathcal{P}(G)$.
\end{proof}

\begin{remark}
The construction of the Whitney decomposition that is described
in Section~\ref{s.notation} yields
%  Using the Whitney decomposition and chains of $\lambda$-dilated
 % cubes, where $1<\lambda<\frac{5}{4}$, as described in
   %we obtain 
   Theorem~\ref{t.local-to-global}
  for all $1<\lambda<\frac{10}{9}$. 
 A simple modification of the definition for dilated cubes $Q^*$ allows
 one to extend this range to every $1<\lambda < \frac{5}{4}$.
  It possible to use the general
  Whitney decomposition based on Stein~\cite[pp. 167--170]{Stein} in
  order to obtain the result for any $\lambda\geq
  \frac{5}{4}$.
\end{remark}

\section{A sufficient condition for a weak-type inequality}\label{s.suff}
\label{sect:weakIneq}

In this section, we show that cubes can be replaced
by  John domains in 
inequality \eqref{eq:weaktype}.

\begin{theorem}\label{t.weak_john}
  Suppose that $G$ is a John domain in $\R^n$. If $f\in L^1(G)$ and
  $1<p<\infty$, then the following weak-type inequality is valid
\begin{equation}\label{e.weak}
\sigma^p \lvert \{x\in G:\ \lvert f(x)-f_G\rvert >\sigma \}\rvert \le 
C\mathcal{K}_{f,\textup{loc}}^p(G)
\end{equation}
for all $\sigma>0$, where a positive constant $C$ depends on $n$, $p$,
$\lambda$, 
and the John constant $\beta_G$.
\end{theorem}

\begin{proof}
  Recall that $Q_0$ is a fixed cube which is used to construct a chain
  decomposition of $G$, see
  Proposition \ref{p.admissible}.
By the triangle inequality for each $x\in G$,
\begin{align*}
  \lvert f(x)-f_{G}\rvert &\le \lvert f_{Q_0^*} - f_G\rvert +
  \left|f(x)-\sum_{Q\in\mathcal{W}(G)} f_{Q^*}\chi_{Q}(x)\right|+
  \left|\sum_{Q\in\mathcal{W}(G)} f_{Q^*} \chi_Q(x) - f_{Q_0^*}\right|
  \\ & =: g_1(x) + g_2(x) + g_3(x)\,.
\end{align*}
Hence, for a fixed $\sigma>0$, we have
\begin{align*}
  \sigma^p \lvert \{x\in G:\ \lvert
  f(x)-f_{G}\rvert>\sigma\}\rvert\le \sigma^p \mathbf{F}_1(\sigma) +
  \sigma^p \mathbf{F}_2(\sigma)+ \sigma^p \mathbf{F}_3(\sigma)\,
\end{align*}
where we have written
\[
\mathbf{F}_j(\sigma)=\left\lvert \{x\in G:\ g_j(x)>\sigma/3\}\right\lvert
\]
for $j\in \{1,2,3\}$. We shall next estimate these three terms. 

If $\lvert f_{Q_0^*} -f_G\rvert \le \sigma/3$, then
$\mathbf{F}_1(\sigma)=0$. Otherwise, by Lemma~\ref{l.local-to-global},
\begin{align*}
  \sigma^p \mathbf{F}_1(\sigma) & \le 3^p \lvert G\rvert \left(\fint_G \lvert
    f(x)-f_{Q_0^*}\rvert\,dx\right)^p \leq C\sum_{Q\in\mathcal{W}(G)}
  \mathcal{K}_f^p(Q^*) \leq C\mathcal{K}_{f,\textup{loc}}^p(G).
\end{align*}
Let us focus on the term $\sigma^p\mathbf{F}_2(\sigma)$. By
applying inequality \eqref{eq:weaktype},
\begin{align*}
  \sigma^p\mathbf{F}_2(\sigma) & = \sum_{Q\in\mathcal{W}(G)} \sigma^p \lvert
  \{x\in \textup{int}(Q)\,:\, g_2(x)>\sigma/3\}\rvert \\&\le \sum_{Q\in\mathcal{W}(G)}
  \sigma^p \lvert \{x\in Q^*:\ \lvert
  f(x)-f_{Q^*}\rvert>\sigma/3\}\rvert 
 \le C 3^p\sum_{Q\in\mathcal{W}(G)} \mathcal{K}_{f}^p(Q^*) \leq C\mathcal{K}_{f,\textup{loc}}^p(G).
\end{align*}
Let us estimate the remaining term $\sigma^p\mathbf{F}_3(\sigma)$ as follows
\begin{align*}
  \sigma^p \mathbf{F}_3(\sigma) & = \sigma^p\sum_{Q\in\mathcal{W}(G)} \lvert
  \{ x\in \textup{int}(Q)\,:\, \lvert f_{Q^*}-f_{Q_0^*}\rvert >\sigma/ 3\}\rvert \\
  &= \sum_{\substack {Q\in\mathcal{W}(G) \\ \lvert
      f_{Q^*}-f_{Q_0^*}\rvert>\sigma/3}}\sigma^p\lvert Q\rvert \le
  3^p \sum_{Q\in\mathcal{W}(G)} \lvert Q\rvert\lvert
  f_{Q^*}-f_{Q_0^*}\rvert^p.
\end{align*}
Estimating as in connection with \eqref{e.sss}, we end up having
\[
\lvert f_{Q^*}-f_{Q_0^*}\rvert^p \le C\left(\sum_{R\in\mathcal{C}(Q)} \fint_{R^*} \lvert 
f(x) - f_{R^*}\rvert\,dx\right)^p.  
\]
We use condition (1) of the chain $\mathcal{C}(Q)$ in
Proposition~\ref{p.admissible}. Then we write for $j\le k+\tau$
\[
1=(\tau+1+k-j)^{-1}(\tau+1+k-j),
\] 
apply H\"older's inequality, and finally use inequality 
\[\sup_{k\in\Z}
\sum_{j=-\infty}^{k+\tau}(\tau+1+k-j)^{-p'}<\infty\,,\] to conclude that
\begin{equation}\label{e.subs}
\begin{split}
  \sigma^p \mathbf{F}_3(\sigma) &\le C\sum_{k=-\infty}^\infty \sum_{Q\in
    \mathcal{W}_k(G)} \lvert Q\rvert\left(\sum_{j=-\infty}^{k+\tau}
    \sum_{\substack{R\in\mathcal{W}_j(G) \\ R\in \mathcal{C}(Q)}}
    \fint_{R^*} \lvert f(x)-f_{R^*}\rvert\,dx
  \right)^p \\
  & \le C\sum_{k=-\infty}^\infty \sum_{Q\in \mathcal{W}_k(G)} \lvert
  Q\rvert\sum_{j=-\infty}^{k+\tau} (\tau+1+k-j)^p
  \left(\sum_{\substack{R\in\mathcal{W}_j(G)\\ R\in \mathcal{C}(Q)}}
    \fint_{R^*} \lvert f(x)-f_{R^*}\rvert\,dx \right)^p.
\end{split}
\end{equation}
By condition (2) %of the chain $\mathcal{C}(Q)$ 
in
Proposition~\ref{p.admissible} and H\"older's inequality,
for any $Q\in\mathcal{W}(G)$ and 
$j\in \Z$,
\begin{equation} \label{e.subs2}
\begin{split}
  \sum_{\substack{R\in\mathcal{W}_j(G) \\ R\in \mathcal{C}(Q)}}
  \fint_{R^*} \lvert f(x)-f_{R^*}\rvert\,dx &\le
  \left(\sum_{\substack{R\in\mathcal{W}_j(G)\\ R\in \mathcal{C}(Q)}}
    1\right)^{1/p'}\left(\sum_{\substack{R\in\mathcal{W}_j(G)\\ R\in
        \mathcal{C}(Q)}}
    \left(\fint_{R^*} \lvert f-f_{R^*}\rvert\right)^p\right)^{1/p} \\
  &\le C\left(\sum_{\substack{R\in\mathcal{W}_j(G)\\ R\in
        \mathcal{C}(Q)}}\frac{\mathcal{K}_f^p(R^*)}{\lvert
      R^*\rvert}\right)^{1/p}.
\end{split}  
\end{equation}
If we substitute the estimate obtained in \eqref{e.subs2} to
\eqref{e.subs}, and observe that $R\in\mathcal{C}(Q)$ if and only if
$Q\in\mathcal{S}(R)$, we bound $\sigma^p \mathbf{F}_3(\sigma)$ as follows
\begin{align*}
  \sigma^p\mathbf{F}_3(\sigma) & \leq C\sum_{k=-\infty}^\infty \sum_{Q\in
    \mathcal{W}_k(G)} \lvert Q\rvert\sum_{j=-\infty}^{k+\tau} (\tau+1+k-j)^p
  \sum_{\substack{R\in\mathcal{W}_j(G) \\ R\in
      \mathcal{C}(Q)}}\frac{\mathcal{K}_f^p(R^*)}{\lvert
    R^*\rvert} \\
%  & = C\sum_{j=0}^\infty \sum_{R\in \mathcal{W}_j(G)}
%  \mathcal{K}_f^p(R^*)\mathbf{T}_{j,R}, \\
  & = C\sum_{j=-\infty}^\infty \sum_{R\in \mathcal{W}_j(G)}
  \frac{\mathcal{K}_f^p(R^*)}{\lvert R\rvert} \sum_{k=j-\tau}^\infty
  \sum_{\substack{Q\in\mathcal{W}_k(G) \\ Q\in \mathcal{S}(R)}}
  \lvert Q\rvert(\tau+1+k-j)^p \\
  & \leq C\sum_{j=-\infty}^\infty \sum_{R\in \mathcal{W}_j(G)}
  \mathcal{K}_f^p(R^*) \leq C\mathcal{K}_{f,\textup{loc}}^p(G),
\end{align*}
%where the constants $\mathbf{T}_{j,R}$ are given by
%\[
% \sum_{k=\max\{0,j-\tau\}}^\infty \sum_{\substack{Q\in\mathcal{W}_k
%     \\ Q\in \mathcal{S}(R) }}
%\lvert Q\rvert\cdot \lvert R^*\rvert^{-1}(\tau+1+k-j)^p\,,\qquad j\ge 0\,,R\in\mathcal{W}_j\,,
%\]
where we used  condition (3) in Proposition~\ref{p.admissible}. The claim follows.
\end{proof} 

We formulate the preceding theorem
for locally integrable functions; the proof is otherwise the same, but term $g_1$ is omitted
and we choose $c=f_{Q_0^*}$.

\begin{theorem}\label{t.weak_local_john}
Suppose that $G$ is a John domain in $\R^n$. If $f\in L^1_{\textup{loc}}(G)$ and
  $1<p<\infty$, then the following weak-type inequality is valid
\begin{equation}\label{e.weak_local}
\inf_{c\in\R} \sup _{\sigma>0} \sigma^p \lvert \{x\in G:\ \lvert f(x)-c\rvert >\sigma \}\rvert \le 
C\mathcal{K}_{f,\textup{loc}}^p(G)\,,
\end{equation}
 where a positive constant $C$ depends on $n$, $p$, $\lambda$, 
and the John constant $\beta_G$.
\end{theorem}

\section{Necessary conditions for a weak-type inequality}
\label{sect:Necessary}

We study necessary conditions for the validity of weak-type inequality
\eqref{e.weak_local} on domains.  In Theorem \ref{t.necessity}, a
necessary condition is formulated in terms of a Poincar\'e inequality.
Corollary \ref{c.nec} addresses the necessity of the John condition.

\begin{theorem}\label{t.necessity}
  Suppose that $n/(n-1)\le p <\infty$, and that $G$ is a bounded domain in
  $\R^n$, $n\ge 2$, for which the inequality
\begin{equation}\label{t.pi}
\inf_{c\in\R} \sup_{\sigma>0} \sigma^p \lvert \{ x\in G:\ \lvert f(x)- c\rvert >\sigma\} \rvert  
\le C\mathcal{K}_{f,\textup{loc}}^p(G)
\end{equation}
holds for all  $f\in L^1_{\textup{loc}}(G)$. Then $G$
satisfies the $(q^*,q)$-Poincar\'e inequality \eqref{e.poincare} with $p=q^*=nq/(n-q)$,
where $1\leq q <n$.
\end{theorem}

\begin{proof}
It is enough to verify that $G$ satisfies the {\em weak}
  $(q^*,q)$-Poincar\'e inequality. That is,  for all
  locally Lipschitz functions $f$ in $G$,
\begin{equation}\label{weak}
\inf_{c\in\R}\sup_{\sigma>0}  \sigma^{q^*}  \lvert \{ x\in G:\ \lvert f(x)- c\rvert >\sigma \} \rvert  \le 
  C\left(\int_{G} \lvert \nabla f(x)\rvert^q\,dx\right)^{q^*/q}\,.
\end{equation}
By applying inequality \eqref{weak} and the Maz'ya truncation
  method, we refer to \cite[Theorem 4]{H}, we may conclude that $G$
  satisfies the $(q^*,q)$-Poincar\'e inequality:
\begin{equation}\label{e.poincare}
\int_G \lvert f(x)-f_G\rvert^{q^*} \,dx \le C\left(\int_G \lvert
\nabla f(x)\rvert^q\,dx\right)^{q^*/q},
\end{equation}
where $f$ is in the Sobolev space $W^{1,q}(G)$.

Therefore, let us prove inequality \eqref{weak}. This will be
 a consequence of the $(q^*,q)$-Poincar\'e inequality on cubes in $G$. 
 Namely,
there is a local
partition $\mathcal{P}_{\textup{loc}}(G)$ such that
\begin{align*}
\inf_{c\in\R} \sup_{\sigma>0}  \sigma^{q^*} \lvert \{ x\in G:\ \lvert f(x)- c\rvert >\sigma\} \rvert  
  %\le C\mathcal{K}_{f,\textup{loc}}^{q^*}(G)\\
  & \le C\sum_{Q\in\mathcal{P}_{\textup{loc}}(G)} \lvert Q\rvert \bigg(\fint_Q
  \lvert f(x)-f_Q\rvert\,dx\bigg)^{q^*} \\
  &\le C\sum_{Q\in\mathcal{P}_{\textup{loc}}(G)} \int_Q \lvert
  f(x)-f_Q\rvert^{q^*}\,dx \\
  & \le C\sum_{Q\in\mathcal{P}_{\textup{loc}}(G)} \left(\int_Q \lvert \nabla
  f(x)\rvert^q \,dx\right)^{q^*/q}.
\end{align*}
Since $q^*/q=n/(n-q)>1$, we obtain
the desired inequality \eqref{weak}.
\end{proof}

\begin{remark}\label{r.setup}
%As a  
%consequence of inequality \eqref{t.pi}, 
We may also conclude  the following
weak fractional Sobolev--Poincar\'e inequality. 
%It is known that  Sobolev--Poincar\'e inequality of order $\delta\in (0,1]$
%is valid on any John domain, we refer e.g. to \cite{B,H-SV}.
%To illustrate, the order $\delta=1$ inequality on a given domain $G$ is
%the following:  for all $f\in W^{1,q}(G)$, $1\le q<n$,
%\[
%\bigg(\int_G \lvert f(x)- f_G\rvert^{q^*}\,dx\bigg)^{1/q^*}\le C\bigg(\int_G \lvert \nabla f(x)\rvert^q
%\,dx\bigg)^{1/q}\,.
%\]
%Inequality \eqref{t.pi} is
%a weak Sobolev--Poincar\'e inequality of order zero, and
%Theorem \ref{t.weak_local_john} implies that this 
% inequality holds on John domains.
%Among the weak Sobolev--Poincar\'e inequalities, the order zero inequality is 
%the strongest.
%We refer to Theorem \ref{t.necessity}. And, concerning the remaining case of $\delta\in (0,1)$, we 
Suppose  that
inequality \eqref{t.pi} holds for all $f\in L^p_{\textup{loc}}(G)$ with
$n/(n-\delta)<p<\infty$ and $\delta\in (0,1)$.  Then the inequality
\begin{equation}\label{e.weakf}
\inf_{c\in\R}\sup_{\sigma>0} \sigma^{q^{*,\delta}} \lvert \{x\in G\,:\, \lvert f(x)-c\rvert>\sigma\}\rvert \le
C\bigg(\int_G\int_G \frac{\lvert f(x)-f(y)\rvert^q}{\lvert x-y\rvert^{n+\delta q}}\,dy\,dx\bigg)^{q^{*,\delta}/q}
\end{equation}
holds for all $f\in L^p_{\textup{loc}}(G)$,
where $p=q^{*,\delta}=nq/(n-\delta q)$ and $1<q<n/\delta$.
Indeed, by proceeding as in the proof of Theorem \ref{t.necessity}, we obtain a local partition 
$\mathcal{P}_{\textup{loc}}(G)$ such that
\begin{equation}\label{e.sec}
\inf_{c\in\R}\sup_{\sigma>0} \sigma^{q^{*,\delta}} \lvert \{x\in G\,:\, \lvert f(x)-c\rvert>\sigma\}\rvert
\le C\sum_{Q\in\mathcal{P}_{\textup{loc}}(G)} \int_Q \lvert
  f(x)-f_Q\rvert^{q^{*,\delta}}\,dx\,.
\end{equation}
The following fractional Sobolev--Poincar\'e inequality
\[
\int_Q \lvert
  f(x)-f_Q\rvert^{q^{*,\delta}}\,dx\le C\bigg(\int_Q\int_Q \frac{\lvert f(x)-f(y)\rvert^q}{\lvert x-y\rvert^{n+\delta q}}\,dy\,dx\bigg)^{q^{*,\delta}/q}\,,
  %\qquad Q\in \mathcal{P}_{\textup{loc}}(G)\,,
  \]
  where $Q\in \mathcal{P}_{\textup{loc}}(G)$ and $C$ is a constant
  depending on $n$, $q$, and $\delta$, is a consequence of a
  translation and scaling argument in combination with \cite[Theorem
  4.10]{H-SV} applied in the John domain $(0,1)^n$.  In particular,
  the right hand side of \eqref{e.sec} is bounded by
\[
C\sum_{Q\in\mathcal{P}_{\textup{loc}}(G)}\bigg(\int_Q\int_Q \frac{\lvert f(x)-f(y)\rvert^q}{\lvert x-y\rvert^{n+\delta q}}\,dy\,dx\bigg)^{q^{*,\delta}/q}\,.
\]
Since $q^{*,\delta}/q> 1$ we can take the summation inside the parentheses
and the proof of the weak type inequality \eqref{e.weakf} is finished 
by recalling that $\sum_{Q\in\mathcal{P}_{\textup{loc}}(G)} \chi_Q\le N\chi_G$. \color{black}
\end{remark}

 We recall from
\cite[Definition 3.2]{BK} that a domain $G$ with a fixed point $x_0$
satisfies a separation property if there exists a constant $C_0$ such
that for each $x\in G$ there is a curve $\gamma$ joining $x$ and $x_0$
in $G$ so that for each $t$ either
\[\gamma([0,t])\subset
B:=B(\gamma(t),C_0\dist(\gamma(t),\R^n\setminus G))\] or each
$y\in\gamma([0,t])\setminus \overline{B}$ belongs to a different
component of $G\setminus\partial B$ than $x_0$.  As an example, for
simply connected planar domains, the separation property is
automatically valid.

The following corollary is a consequence of Theorem~\ref{t.weak_local_john},
Theorem~\ref{t.necessity}, and \cite[Theorem 1.1]{BK}.

\begin{corollary}\label{c.nec}
  Suppose that $G$ is a bounded domain in $\R^n$, $n\ge 2$, satisfying a
  separation property. Assume further that $n/(n-1)\leq
  p<\infty$. Then the weak-type inequality
\begin{equation}
\inf_{c\in\R}\sup_{\sigma>0} \sigma^p \lvert \{ x\in G:\ \lvert f(x)- c\rvert >\sigma\} \rvert  
\le C\mathcal{K}_{f,\textup{loc}}^p(G)
\end{equation}
holds for every $f\in L^1_{\textup{loc}}(G)$ if, and
only if, $G$ is a John domain.
\end{corollary}

%\begin{question}
%To our knowledge, it is an open question whether the strong Sobolev--Poincar\'e inequality of order 
%$\delta\in [0,1)$ on a given domain
%is a consequence of the corresponding weak inequality.
%In case of $\delta=1$ this result is known to be true,  we refer to the proof of Theorem \ref{t.necessity}.
%\end{question}


\begin{thebibliography}{99}

\bibitem{ABKY} \art{Aalto, D, Berkovits, L., Kansanen, O. E., \AND
    Yue, H.}{John--Nirenberg lemmas for a doubling measure} {Studia
  Math.}{204}{2011}{21--37}

\bibitem{BM} {\sc Berkovits, L., Kinnunen, J., \AND Martell, J. M.},
  {Self-improvement and good-$\lambda$ inequalities}, {\it Preprint}.

\bibitem{Boman} {\sc Boman, J.}, {$L_p$-estimates for very strongly
    elliptic systems}, {Report no. 29, Department of Mathematics,
    University of Stockholm, Sweden, 1982.}

\bibitem{BK} \art{Buckley, S. \AND Koskela, P.}
  {Sobolev--Poincar\'e implies John} 
  {Math. Res. Lett.}{2}{1995}{577--593}

\bibitem{FPW} \art{Franchi, B., P\'erez, C., \AND Wheeden, R. L.}
  {Self-improving properties of John-Nirenberg and Poincar\'e
    inequalities on spaces of homogeneous
    type}{J. Funct. Anal.}{153}{1998}{108--146}

\bibitem{FPW03} \art{Franchi, B., P\'erez, C., \AND Wheeden, R. L.}
  {A sum operator with applications to self-improving properties of
    Poincar\'e inequalities in metric spaces} {J. Fourier
    Anal. Appl.}{9}{2003}{511--540}

\bibitem{Giusti} \book{Giusti, E.}  {Direct Methods in the Calculus of
    Variations}{World Scientific Publishing Co., Inc., River Edge, NJ,
    2003}

\bibitem{H} \art{ Haj\l asz, P.} {Sobolev inequalities, truncation
    method, and John domains} {Papers on analysis,
    Reports Univ. Jyv\"askyl\"a, Dep. Math. Stat.} {83}
    {2001} {109--126}
    
  \bibitem{HK} \art{ Haj\l asz, P., \AND Koskela. P.} {Isoperimetric
      inequalities and imbedding theorems in irregular domains}
    {J. London Math. Soc.} {58} {1998} {425--450}

\bibitem{H-SV} \art{Hurri-Syrj{\"a}nen, R. \AND V{\"a}h{\"a}kangas, A. V.}
{On fractional Poincar{\'e} inequalities}{J. Anal. Math.}{120}{2013}{85--104}

\bibitem{J} \art{John, F.}  {Rotation and strain} {Comm. Pure
    Appl. Math.} {14} {1961} {391--413}

\bibitem{JN} \art{John, F. \AND Nirenberg, L.}
  {On functions of bounded mean oscillation}
  {Comm. Pure Appl. Math.}{14}{1961}{415--426}

\bibitem{lehrbackII} \art{Lehrb\"ack, J. \AND Tuominen, H.}
  {A note on the dimensions of Assouad and Aikawa}
  {J. Math. Soc. Japan}{65}{2013}{343--356}

\bibitem{Luukkainen} \art{Luukkainen, J.}  {Assouad dimension:
    antifractal metrization, porous sets, and homogeneous measures}
  {J. Korean Math. Soc.}{35}{1998}{23--76}

\bibitem{MacP} \art{MacManus, P. \AND P\'erez, C.}
  {Generalized Poincar\'e inequalities: sharp self-improving properties}
  {Internat. Math. Res. Notices}{2}{1998}{101--116}

\bibitem{MacP2} \art{MacManus, P.  \AND P\'erez, C.}
{Trudinger inequalities without derivatives} 
{Trans. Amer. Math. Soc.}{354}{2002}{1997--2012}

\bibitem{RR} \book{Reimann, H. M. \AND Rychener, T.}  {Funktionen
    Beschr\"ankter Mittlerer Oszillation} {Lecture Notes in
    Mathematics, Vol. 487. Springer-Verlag, Berlin-New York, 1975}

\bibitem{Stein} \book{Stein, E. M.}{Singular Integrals and
    Differentiability Properties of Functions}{Princeton University
    Press, Princeton, NJ, 1970}

\bibitem{Tor} \book{Torchinsky, A.}  {Real-variable Methods in
    Harmonic Analysis}{Dover Publications, Inc., Mineola, NY, 2004}

\end{thebibliography}
\end{document}